\documentclass[review]{elsarticle}

\usepackage{lineno,hyperref}
\usepackage{comment}
\usepackage{tikz}
\usepackage{longtable}
\usepackage{float}
\usepackage{amssymb}  
\usepackage{url}  

\newcommand{\FLT}[1]{{\rm FLT}_{#1}}
\newcommand{\into}{\rightarrow}

\newcommand{\Z}{\mathbb{Z}}
\newcommand{\N}{\mathbb{N}}
\newcommand{\Q}{\mathbb{Q}}

\newcommand{\COL}{{\rm COL}}
\newcommand{\I}{{\rm I}}
\newcommand{\D}{{\rm D}}
\newcommand{\K}{{\rm K}}
\newcommand{\U}{{\rm U}}
\newcommand{\st}{{\ \colon \ }}

\newcounter{savenumi}

\newtheorem{theoremfoo}{Theorem}[section] 
\newenvironment{theorem}{\pagebreak[1]\begin{theoremfoo}}{\end{theoremfoo}}

\newtheorem{lemmafoo}[theoremfoo]{Lemma}
\newenvironment{lemma}{\pagebreak[1]\begin{lemmafoo}}{\end{lemmafoo}}

\newtheorem{notefoo}[theoremfoo]{Note}

\newtheorem{conjecturefoo}[theoremfoo]{Conjecture}
\newenvironment{conjecture}{\pagebreak[1]\begin{conjecturefoo}}{\end{conjecturefoo}}

\newtheorem{corollaryfoo}[theoremfoo]{Corollary}
\newenvironment{corollary}{\pagebreak[1]\begin{corollaryfoo}}{\end{corollaryfoo}}

\newtheorem{notationfoo}[theoremfoo]{Notation}
\newenvironment{notation}{\pagebreak[1]\begin{notationfoo}}{\end{notationfoo}}

\newcommand{\fig}[1] 
{
\begin{figure}
\begin{center}
\input{#1}
\end{center}
\end{figure}
}

\newtheorem{dfntn}[theoremfoo]{Def}
\newenvironment{definition}{\pagebreak[1]\begin{dfntn}\rm}{\end{dfntn}}

\newenvironment{proof}
   {\pagebreak[1]{\narrower\noindent {\bf Proof:\quad\nopagebreak}}}{\QED}

\newcommand{\yyskip}{\penalty-50\vskip 5pt plus 3pt minus 2pt}
\newcommand{\blackslug}{\hbox{\hskip 1pt
       \vrule width 4pt height 8pt depth 1.5pt\hskip 1pt}}
\newcommand{\QED}{{\penalty10000\parindent 0pt\penalty10000
       \hskip 8 pt\nolinebreak\blackslug\hfill\lower 8.5pt\null}
       \par\yyskip\pagebreak[1]}

\newcommand{\BLUEB}{{\penalty10000\parindent 0pt\penalty10000
       \hskip 8 pt\nolinebreak\hbox{\ }\hfill\lower 8.5pt\null}
       \par\yyskip\pagebreak[1]}

\modulolinenumbers[5]

\journal{Discrete Mathematics}

\begin{document}

\begin{frontmatter}

\title{ Fermat's Last Theorem, Schur's Theorem (in Ramsey Theory), and the Infinitude of the Primes}



\author[billmainaddress]{William Gasarch\corref{mycorrespondingauthor}}
\cortext[mycorrespondingauthor]{Corresponding author}
\ead{gasarch@umd.edu}

\address[billmainaddress]{Univ.\ of MD at College Park, Dept of Comp.\ Sci.}

\begin{abstract}
Alpoge and Granville (separately) gave novel proofs that the primes are infinite that use
Ramsey Theory. In particular, they use Van der Waerden's Theorem and some number theory.
We prove the primes are infinite using an easier theorem from Ramsey Theory, namely
Schur's Theorem, and some number theory (Elsholtz independently obtained the same proof that
the primes were infinite). In particular, we use the $n=3$ case of
Fermat's last theorem. 
We also apply our method to show other domains have an infinite number of irreducibles.
\end{abstract}

\begin{keyword}
\texttt{Primes; Ramsey Theory; Schur's Theorem; Fermat's Last Theorem; Irreducibles; Colorings}
\MSC[2020] 11A41, 05D10
\end{keyword}

\end{frontmatter}


\section{Introduction}

\begin{notation}
We take $\N$ to be $\{0,1,2,3,\ldots\}$.
\end{notation} 

\begin{definition}
Let $a\in\N$ and $\D$ be a domain.
\begin{enumerate}
\item 
{\it $\FLT a$ holds in $\D$} means that 
the equation 

$$x^a + y^a = z^a$$

\noindent
has no solution in $\D-\{0\}$. 
\item
$\FLT a$ means $\FLT a$ holds in $\Z$. 
\end{enumerate}
\end{definition}

In 1770 Euler proved $\FLT 3$ (see the texts of Ireland \& Rosen~\cite{IR-1982} or Hardy \& Wright~\cite{HR-1979}
for a modern treatment of Euler's proof). 
In 1916 Schur proved a theorem in Ramsey Theory (which we will state later) 
that is referred to as {\it Schur's Theorem (in Ramsey Theory)}
(see the texts of Graham-Rothschild-Spencer~\cite{GRS-1990} or Landman \& Robertson~\cite{LR-2004} for 
a modern treatment of Schur's proof). 
In this paper we use these two theorems
to prove the primes are infinite. 
(Elsholtz~\cite{Elsholtz-2021} independently obtained
the same proof that the primes were infinite.) 
While there are of course easier proofs, we think it is
of interest that it can be derived from Schur's Theorem and $\FLT 3$. 

Alpoge~\cite{Alpoge-2015}
proved the primes were infinite using elementary number theory and 
Van der Warden's theorem.
Granville~\cite{Granville-2017}
proved that the primes were infinite from 
the fact that that there can never be four squares in arithmetic progression (attributed to Fermat)
and 
Van der Warden's theorem.
Our proof compares to their proofs as follows:
\begin{itemize}
\item
Our proof uses easier Ramsey Theory then Alpoge's or Granville's proof.
\item 
Our proof uses harder number theory than Alpoge's proof.
\item
Our proof uses about the same level of number theory as Granville's proof.
\item
We prove a general theorem that allows us to show other domains have an infinite
number of irreducibles. 
\end{itemize}

In Section~\ref{se:pre} we present Schur's Theorem and definitions from number theory.
In Section~\ref{se:link} we present a condition on an integral domains $\D$ that implies
$\D$ has an infinite number of irreducibles.
That condition easily applies to $\Z$. Hence we obtain that $\Z$ has an infinite number of
irreducibles. Since in $\Z$, every irreducible is a prime, we also get that there are an infinite
number of primes.
In Section~\ref{se:sqrt} 
we use our results to show that, for all $d\in\N$,  $\Z[\sqrt{-d}]$ 
has an infinite number of irreducibles.
In Section~\ref{se:conj} we use our results, together with a widely believed conjecture,
to show that many domains have an infinite number of irreducibles.
In Section~\ref{se:open} we present an open problem.


\section{Preliminaries}\label{se:pre}

The following is Schur's Theorem (from Ramsey theory). It can be proven from Ramsey's Theorem.

\begin{lemma}\label{le:schur}
For all $c$, for all $c$-colorings $COL:\N-\{0\}\into [c]$, there exist $x,y,z\in\N-\{0\}$ with 
$x+y=z$ such that 
$$COL(x)=COL(y)=COL(z).$$
\end{lemma}

The following definitions are standard. 

\begin{definition}
Let $\D$ be an integral domain.
\begin{enumerate}
\item
A {\it unit} is a $u\in \D$ such that there exists $v\in \D$ with
$uv=1$.
We let $\U$ be the set of units
if the domain is understood.
\item
An {\it irreducible} is a $p\in \D-\U$ such that if $p=ab$ then
either $a\in \U$ or $b\in \U$.
We let $\I$ be the set of irreducibles 
if the domain is understood.
\item
A {\it prime} is a $p\in\D$ such that if $p$ divides $ab$ then
either $p$ divides $a$ or $p$ divides $b$.
In any integral domain all primes are irreducible.
There are integral domains with irreducibles that are not primes.
The set

\noindent
 $\{a+b\sqrt{-5} \st a,b\in \Z\}$ is one such example:
(a) The element 2 is irreducible, yet (b) 2 is not prime since
2 divides $(1+\sqrt{-5})(1-\sqrt{-5})=6$ but 2 does not divide 
either $1+\sqrt{-5}$ or $1+\sqrt{-5}$. 
\item
We impose an equivalence relation on $\I$: $p$ and $q$ are equivalent
if there exists $u\in \U$ such that $p=uq$.
We say {\it $\I$ is infinite up to units} if the number of equivalence
classes is infinite. In this paper {\it infinite} will mean {\it infinite up to units}. 
\item
An {\it Atomic Integral Domain} is an integral domain such that every element of $\D-(\U\cup \{0\})$
can be written (not necessarily uniquely) as $p_1^{x_1}\cdots p_m^{x_m}$
where the $p_i$'s are irreducible. 
The domains $\Z$ and $\Z[\sqrt{d}]$ are known to be atomic by using norms.
The set of algebraic integers (complex numbers that satisfy monic polynomials over $\Z[x]$)
is an integral domain that is not atomic for a funny reason: there are no
irreducibles. If 
$a$ is a nonzero nonunit algebraic integer then $\sqrt a$ is a nonzero nonunit algebraic 
integer, and $a=\sqrt{a}\times\sqrt{a}$, so $a$ is not irreducible. 
\end{enumerate}
\end{definition}

\section{A Condition for a Domain to Have an Infinite Number of Irreducibles}\label{se:link}

Theorem~\ref{th:link} says that if an integral domain $\D$ has a finite number of
irreducibles then an equation similar to that in FLT has a solution. 
We will use Theorem~\ref{th:link} to derive conditions on $\D$ that imply
it has an infinite number of irreducibles. 

The coloring in the proof of Theorem~\ref{th:link} is similar to the one
used by Alpoge~\cite{Alpoge-2015},
Granville~\cite{Granville-2017}, and
Elsholtz~\cite{Elsholtz-2021}.

\begin{theorem}\label{th:link}~
Let $\D$ be an atomic integral domain that contains $\N$. 
Assume there exists an $n\ge 2$ such that the following equation has no solution:
$$u_xX^n+u_yY^n=u_zZ^n$$
\noindent
where $u_x,u_y,u_z\in \U$ and $X,Y,Z\in \D-\{0\}$. 
Then $\D$ has an infinite number of irreducibles.
\end{theorem}

\begin{proof}
Assume the premise is true.
Assume, by way of contradiction, that $\I$ is finite. 
Let $\I=\{p_1,\ldots,p_m\}$ be formed by taking an irreducible from each equivalence class.

Since $\D$ is atomic, every $x\in\D-\{0\}$ can be written as
$u p_1^{x_1} \cdots p_m^{x_m}$ where $u\in \U$ and $x_1,\ldots,x_m\in\N$
(an $x_i$ can be 0). 
This need not be unique; however, for the sake of definiteness, we will take 
$(x_1,\ldots,x_m)$ to be the lexicographically least tuple.

Recall that $\N\subseteq \D$. 
Let $n$ be as in the premise. 
We define a coloring $\COL$ of $\N-\{0\}$ 
as  follows: 
Color $x=u p_1^{x_1}\cdots p_m^{x_m}$ by the vector 
$$(x_1 \bmod n,\ldots, x_m \bmod n).$$
There are $n^m$ colors, which is finite.
By Lemma~\ref{le:schur} there exists  $(x,y,z)$, and a color 
$(e_1,\ldots,e_m)$, such that

$$\COL(x)=\COL(y)=\COL(z)=(e_1,\ldots,e_m).$$

and

$$x+y=z.$$

We now reason about $x$ but the same logic applies to $y,z$. 
Note that there exist $u\in \U$ and $k_1,\ldots,k_m\in\N$ such that

$$x= u p_1^{k_1n+e_1} \cdots p_m^{k_mn+e_m}$$

hence

$$x p_1^{n-e_1} \cdots p_m^{n-e_m} = u p_1^{(k_1+1)n} \cdots p_m^{(k_m+1)n}=uX^n$$

\noindent
where $X=p_1^{(k_1+1)} \cdots p_m^{(k_m+1)}\in \D$.

Since the same logic applies to $y,z$ we have that 
there exist $X,Y,Z\in \D$ and $u_x,u_y,u_z\in \U$ such that

$x p_1^{n-e_1} \cdots p_m^{n-e_m}=u_xX^n$

$y p_1^{n-e_1} \cdots p_m^{n-e_m}=u_yY^n$

$z p_1^{n-e_1} \cdots p_m^{n-e_m}=u_zZ^n$.

Note that the following hold:
\begin{itemize}
\item 
$u_xX^n+u_yY^n=u_zZ^n$.
\item 
$u_x,u_y,u_z\in \U$.
\item 
$X,Y,Z\in \D-\{0\}$. 
\end{itemize}

This contradicts the premise of the theorem. 
\end{proof}

\begin{theorem}\label{th:strong}
Let $\D$ be an atomic integral domain. 

\begin{enumerate}

\item
Assume that there is an $n_0\in\N$, $n_0\ge 2$,  such that
the following hold:
\begin{itemize}
\item 
For all $u\in \U$, there is $v\in \D$ such that $v^{n_0}=u$.
\item 
$\FLT {n_0}$ holds for $\D$. 
\end{itemize}
Then $\D$ has an infinite number of irreducibles.

\item
Assume that there is an $n_0\in\N$, $n_0\ge 2$,  such that
the following hold:
\begin{itemize}
\item
For all $u\in \U$, $u^{n_0}=u$.
\item
$\FLT {n_0}$ holds for $\D$. 
\end{itemize}
Then $\D$ has an infinite number of irreducibles.
(This follows from Part 1.) 

\end{enumerate}
\end{theorem}

\begin{proof}

Assume, by way of contradiction, that $\D$ has a finite number of irreducibles.
By Theorem~\ref{th:link}, for all $n\in \N$
there exist $u_x,u_y,u_z\in\U$ and $X,Y,Z\in \D-\{0\}$ such that the following holds:

$$u_xX^n+u_yY^n=u_zZ^n.$$

Take $n=n_0$. By the first premise, there exists $v_x, v_y, v_z$ such that 
$v_x^{n_0} = u_x$, 
$v_y^{n_0} = u_y$, 
$v_z^{n_0} = u_z$.
Hence 

$$(v_xX)^{n_0}+(v_yY)^{n_0}=(v_zZ)^{n_0}.$$

By the second premise, that $\FLT {n_0}$ holds for $\D$, this is a contradiction.
\end{proof}

\begin{corollary}\label{co:Z}~
\begin{enumerate}
\item
$\Z$ has an infinite number of irreducibles.
\item
$\Z$ has an infinite number of primes.
\end{enumerate}
\end{corollary}

\begin{proof}

\noindent
1) Let $n=3$. The only units in $\Z$ are $\{-1,1\}$. Note that (a) all $u\in \{-1,1\}$ 
satisfy $u^3=u$, and (b) $\FLT 3$ holds for $\Z$. 
Hence, by Theorem~\ref{th:strong}.2, $\Z$ has an infinite number of irreducibles.

\smallskip

\noindent
2) In $\Z$ all irreducibles are primes. Hence $\Z$ has an infinite number of primes.

\end{proof}

\section{A Sanity Check} 

As a sanity check on Theorem~\ref{th:link} we look at two integral domains that
have a {\it finite} number of irreducibles.

\begin{enumerate}
\item 
Consider $\Q$.
Note that $\U=\Q-\{0\}$, so there are no irreducibles.
Fix $n\ge 3$.
The premise of Theorem~\ref{th:link} does not hold. 
For all $n$ there is a solution to 
$$u_xX^n+u_yY^n=u_zZ^n$$
\noindent
with $u_x,u_y,u_z\in \U$, namely $u_x=u_y=\frac{1}{2}$, $u_z=1$, $X=Y=Z=1$. 

\item 
In this example the variables $a,b,c,d$ are always in $\Z$. 
Let $\D$ be the domain with set 

$$\biggl \{ \frac{a}{b} \st b\equiv 1 \pmod 2  \biggr \}.$$
Clearly  

$$\U=\biggl \{\frac{a}{b} \st a,b\equiv 1 \pmod 2\biggl \}.$$

We show that $\I=\{2\}$. Recall that what we really mean is that 
all irreducibles are of the form $2u$ where $u\in \U$. 

The nonzero elements that are not in $\U$ are in one of the following sets. 

\begin{enumerate}
\item
$\{ \frac{2c}{b} \st c\equiv 1 \pmod 2,\ \  b\equiv 1 \pmod 2 \}.$
Since $\frac{c}{b}\in \U$, these elements are irreducibles in the same
equivalence class as 2. 
\item 
$\{ \frac{2^d c}{b} \st  d\ge 2, c\equiv 1 \pmod 2,\ \  b\equiv 1 \pmod 2 \}.$
These elements are reducible since 
$\frac{2^d c}{b}=2 \times \frac{2^{d-1}c}{b}$ and, since $d\ge 2$, 
$\frac{2^{d-1}c}{b}$ is not a unit. 

\end{enumerate} 

We  must now see how $\D$ violates the premise of Theorem~\ref{th:link}.
We need to show that,
for all $n\in\N$, there is a solution to 
$$u_xX^n+u_yY^n=u_zZ^n$$
\noindent
with $u_x,u_y,u_z\in \U$. 

For $n=1$ we can take $u_x=u_y=u_z=X=Y=1$ and $Z=2$. 
For $n\ge 2$ we can take $u_x=2^{n-1}-1$, $u_y=2^{n-1}+1$, $X=Y=1$, $Z=2$. 
\end{enumerate}

\section{The Domain $\Z[\sqrt{-d}]$ Has an Infinite Number of Irreducibles}\label{se:sqrt}

\begin{lemma}\label{le:units}
Let $d\in\N$.
\begin{enumerate}
\item
If $d=1$ 
then the only units in $\Z[\sqrt{-d}]$
are $\{-1,1,-i,i\}$
\item
If $d\ge 2$ 
then the only units in $\Z[\sqrt{-d}]$
are $\{-1,1\}$
\item
If $d\in\N$ and $u$ is a unit of $\Z[\sqrt{-d}]$ then
$u^9=u$ (This follows from Part 1 and 2. It is also the case
that $u^5=u$; however, 9 is useful to us and, alas, 5 is not) 
\end{enumerate}
\end{lemma}

\begin{proof}

Let $N$ be the standard norm

$$N(a+b\sqrt{-d}) = (a+b\sqrt{-d})(a-b\sqrt{-d}) = a^2 + b^2d.$$

It is well known and easy to verify that $N(xy) = N(x)N(y)$.

If $a_1+b_1\sqrt{-d}$ is a unit then there exist $a_2,b_2$
such that

$$(a_1+b_1\sqrt{-d})(a_2+b_2\sqrt{-d})=1$$

Take the norm of both sides to get

$$(a_1^2+b_1^2d)(a_2^2+b_2^2d)=1$$

Since squares are positive we have that
$a_1^2+b_1^2d=1$.

If $d=1$ then we have $a_1^2+b_1^2=1$, so
$(a_1,b_1)$ is either $(1,0)$, $(-1,0)$, $(0,1)$,
or $(0,-1)$. This yields units
$\{-1,1,-i,i\}$

If $d\ge 2$ then $b_1=0$ so the only units are
$-1,1$.
\end{proof}

Aigner~\cite{Aigner-1957} proved the following (see also Ribenbiom~\cite{Ribenbiom-1979}).

\begin{lemma}\label{le:flt9}
For all $d\in \Z$, $\FLT 6$ and $\FLT 9$ hold in $\Q(\sqrt{-d})$ and 
hence in $\Z[\sqrt{-d}]$. (We will only use $\FLT 9$.)
\end{lemma}

\noindent
{\bf Note}
The following counterexamples show why Lemma~\ref{le:flt9} does not work for
$\FLT 3$, $\FLT 4$, or $\FLT {6k\pm 1}$. 
As far as we know it is an open problem as to whether Lemma~\ref{le:flt9} is true for 8.
\begin{itemize}
\item
In $\Q(\sqrt{2})$:
$(18+17\sqrt{2})^3 + (18-17\sqrt{2})^3 = 42^3.$
\item
In $\Q(\sqrt{-7})$:
$(1+\sqrt{-7})^4 + (1-\sqrt{-7})^4 = 2^4.$
\item
In $\Q(\sqrt{-3})$:
$(1+\sqrt{-3})^{6k\pm 1} + (1-\sqrt{-3})^{6k\pm 1} = 2^{6k\pm 1}.$
\end{itemize}

\begin{theorem}
Let $d\ge 1$. Then there are an infinite number of irreducibles in $\Z[\sqrt{-d}]$.
\end{theorem}

\begin{proof}
Let $\D=\Z[\sqrt{-d}]$. One can show that $\D$ is atomic using norms. 

Let $n_0=9$.
By Lemma~\ref{le:units}, for all $u\in \U$, $u^{n_0}=u$.
By Lemma~\ref{le:flt9} $\FLT {n_0}$ holds for $\D$.
By Theorem~\ref{th:strong}.2 with $n_0=9$, $\D$ has an infinite number
of irreducibles.
\end{proof}

\section{Conjecturally, Some $\D$ Have an Infinite Number of Irreducibles}\label{se:conj}

Debarre-Klassen~\cite{DK-1994} stated the following conjecture:

\begin{conjecture}\label{co:nf}
Let $\K$ be a number
field of degree $d$ over $\Q$. Let $n\ge d+2$. 
Then $\FLT n$ holds for $\K$.
\end{conjecture}

\begin{theorem}
Assume Conjecture~\ref{co:nf} is true. 
Let $\K$ be a number field of finite degree over $\Q$. 
Let $\D$ be an atomic subdomain of $\K$ with a finite number
of units. Then $\D$ has an infinite number of
irreducibles.
\end{theorem}

\begin{proof}
Let $\K$ and $\D$ be as in the premise.

Since $\D$ has a finite number of units, for each unit $u$,
there exists $n_u$ such that $u^{n_u}=1$.
Let $n_U$ be the lcm of all the $n_u$.
Note that, for all units $u$, $u^{n_U}=1$.
Hence, for all $n\equiv 1 \pmod {n_U}$, $u^n=u$.

Let $n_0$ be such that $n_0\equiv 1 \pmod {n_U}$ and
$n_0\ge d+2$. Then 
(1) $\FLT {n_0}$ holds in $\D$, and 
(2) for all $u\in \U$, $u^{n_0}=u$. 
By Theorem~\ref{th:strong}.2, $\D$ has
an infinite number of irreducibles.
\end{proof}

\section{Open Problem}\label{se:open}
Find other domains to apply Theorem~\ref{th:link} to.
This might involve proving, for fixed $n$,  variants of $\FLT n$ 
that allow units as coefficients.

\section{Acknowledgments}

I thank 
Nathan Cho, 
Emily Kaplitz,
Issac Mammel,
David Marcus,
Adam Melrod, 
Yuang Shen,
Larry Washington, and 
Zan Xu  for proofreading and commentary.
We thank the referees for insightful comments and references
that improved both the readability and correctness of this paper.

\end{document}